\documentclass{amsart}
\usepackage[utf8]{inputenc}
\usepackage{amsthm}

\usepackage{colonequals}
\usepackage[linesnumbered,ruled,vlined]{algorithm2e}
\usepackage{xcolor}
\usepackage{color}
\usepackage[pdfencoding=auto, psdextra, breaklinks=true]{hyperref}
\definecolor{mylinkcolor}{rgb}{0.5,0.0,0.0}
\definecolor{myurlcolor}{rgb}{0.0,0.0,0.7}
\hypersetup{
 colorlinks,
 urlcolor=myurlcolor,
 citecolor=myurlcolor,
 linkcolor=mylinkcolor,
 breaklinks=true,
 pdfauthor={Edgar Costa, Kiran S. Kedlaya, David Roe}
}
\usepackage[capitalize]{cleveref}
\usepackage{multirow}

\usepackage{csquotes} 
\usepackage[style=alphabetic, backend=biber, sorting=nyt, doi=false, isbn=false,url=false, hyperref,backref,backrefstyle=none, maxnames=20, maxalphanames=20]{biblatex}
\newbibmacro{string+doiurlisbn}[1]{%
  \iffieldundef{doi}{%
    \iffieldundef{url}{%
      \iffieldundef{isbn}{%
        \iffieldundef{issn}{%
          #1%
        }{%
          \href{https://books.google.com/books?vid=ISSN\thefield{issn}}{#1}%
        }%
      }{%
        \href{https://books.google.com/books?vid=ISBN\thefield{isbn}}{#1}%
      }%
    }{%
      \href{\thefield{url}}{#1}%
    }%
  }{%
    \href{https://dx.doi.org/\thefield{doi}}{#1}%
  }%
}
\DeclareFieldFormat{title}{\usebibmacro{string+doiurlisbn}{\mkbibemph{#1}}}
\DeclareFieldFormat[article,incollection,inproceedings,unpublished,misc,book]{title}%
    {\usebibmacro{string+doiurlisbn}{\mkbibquote{#1}}}
\DefineBibliographyStrings{english}{%
  backrefpage = {$\uparrow$},%
  backrefpages = {$\uparrow$}%
}
\usepackage{xurl}
\usepackage{breakurl}

\newcommand{\Z}{\mathbb{Z}}
\newcommand{\Q}{\mathbb{Q}}

\newcommand{\Zp}{{\Z}_p}

\newcommand{\Hp}{\mbox{\small $H_p\left(\begin{matrix}\alpha\\\beta\end{matrix}\Big\vert z\right)$}}
\newcommand{\Hpfull}{H_p\left(\begin{matrix}\alpha\\\beta\end{matrix}\Big\vert z\right)}

\newcommand{\gammaprod}{\prod_{\gamma \in \beta}^{\gamma \in \alpha}}

\numberwithin{equation}{section}

\theoremstyle{plain}

\newtheorem{theorem}[equation]{Theorem}
\Crefname{theorem}{Theorem}{Theorems}

\Crefname{claim}{Claim}{Claim}

\Crefname{proposition}{Proposition}{Propositions}
\newtheorem{lemma}[equation]{Lemma}
\Crefname{lemma}{Lemma}{Lemmas}

\Crefname{corollary}{Corollary}{Corollaries}

\Crefname{conjecture}{Conjecture}{Conjectures}

\Crefname{hypothesis}{Hypothesis}{Hypotheses}
\Crefname{section}{\S\!\!}{\S\!\!}

\theoremstyle{definition}
\newtheorem{definition}[equation]{Definition}
\Crefname{definition}{Definition}{Definitions}
\newtheorem{notation}[equation]{Notation}
\Crefname{notation}{Notation}{Notations}
\newtheorem{example}[equation]{Example}
\Crefname{example}{Example}{Examples}

\Crefname{example}{Example}{Examples}

\Crefname{example}{Example}{Examples}

\Crefname{algm}{Algorithm}{Algorithms}

\theoremstyle{remark}
\newtheorem{remark}[equation]{Remark}
\Crefname{remark}{Remark}{Remarks}

\DeclareMathOperator{\Frac}{Frac}
\DeclareMathOperator{\frob}{Frob}
\DeclareMathOperator{\tr}{Tr}

\DeclareMathOperator{\lcm}{lcm}
\DeclareMathOperator{\rad}{rad}

\newcommand{\Sage}{Sage}
\newcommand{\Magma}{Magma}

\usepackage{arydshln} 
\title{Hypergeometric $L$-functions in average polynomial time, II}
\author{Edgar Costa, Kiran S. Kedlaya, and David Roe}
\date{May 2024}

\thanks{Costa and Roe were supported by the Simons Collaboration on Arithmetic Geometry, Number Theory, and Computation via Simons Foundation grant 550033.
Kedlaya was supported by NSF (DMS-1802161, 2053473), UC San Diego (Warschawski Professorship), and the Simons Foundation (Simons Fellows program, 2023--2024); he was also hosted by HIM (Bonn) during summer 2023 and IAS (Princeton) during fall 2023. We thank Will Sawin and Drew Sutherland for helpful discussions.}

\begin{document}

\begin{abstract}
We describe an algorithm for computing, for all primes $p \leq X$, the trace of Frobenius at $p$ of a hypergeometric motive over $\Q$ in time quasilinear in $X$.
This involves computing the trace modulo $p^e$ for suitable $e$;
as in our previous work treating the case $e=1$, we combine the Beukers--Cohen--Mellit trace formula with average polynomial time techniques of Harvey and Harvey--Sutherland.
The key new ingredient for $e>1$ is an expanded version of Harvey's ``generic prime'' construction, making it possible to incorporate certain $p$-adic transcendental functions into the computation; one of these is the $p$-adic Gamma function, whose average polynomial time computation is an intermediate step which may be of independent interest. We also provide an implementation in Sage and discuss the remaining computational issues around tabulating hypergeometric $L$-series.
\end{abstract}

\maketitle

\section{Introduction}

We continue the investigation begun in \cite{costa-kedlaya-roe20} of computational aspects of $L$-functions associated to \emph{hypergeometric motives} in the sense of \cite{roberts-rodriguez-villegas}.
These $L$-functions are easily accessed via the Beukers--Cohen--Mellit trace formula \cite{beukers-cohen-mellit-15} together with the $p$-adic expression of Gauss sums via the Gross-Koblitz formula \cite{gross-koblitz-79}.
While the trace formula has $O(p)$ terms, the main result of \cite{costa-kedlaya-roe20} gives a way to amortize the cost over $p$; that is, one obtains an efficient algorithm for computing, for all primes $p \leq X$, the mod-$p$ reduction of the trace of Frobenius at $p$ of a fixed hypergeometric motive in time quasilinear in $X$.

Here we do the same for the mod-$p^e$ reduction for any positive integer $e$, answering a question raised at the end of \cite{costa-kedlaya-roe20}. Since the trace is an integer in a known range (see \cref{rmk:sufficient precision}), this yields an algorithm for computing the exact trace. (See \cref{def:datum} for
terminology.)
\begin{theorem} \label{T:main}
Algorithm~\ref{alg:Hp overall}, on input of a Galois-stable hypergeometric datum $(\alpha, \beta) \in \Q^r \times \Q^r$, a parameter $z \in \Q \setminus \{0,1\}$, and a positive integer $X$, computes the hypergeometric trace $\Hp$ for all primes $p \leq X$ (excluding tame and wild primes).
Assuming that $r = O(\log X)$ and the bitlength of $z$ is $O(r^2 \log X)$,
the time and space complexities are respectively bounded by
\begin{gather*}
O(r^5 X (\log X)^3) \mbox{ bit operations and } O(r^5 X (\log X)^2) \mbox{ bits.}
\end{gather*}
\end{theorem}

As in \cite{costa-kedlaya-roe20}, the general strategy is to amortize the computational work over primes
using \emph{average polynomial time} techniques of Harvey~\cite{harvey-14, harvey-15} and Harvey--Sutherland \cite{harvey-sutherland-14,harvey-sutherland-16}.
In \cite{costa-kedlaya-roe20}, this amortization is achieved by expressing the trace formula in terms of a series of matrix products of a special form: we have a sequence of matrices defined over $\Z$,
and the desired quantity is obtained by truncating the product at an index depending (linearly) on $p$ and reducing modulo $p$.
One then uses an \emph{accumulating remainder tree}\footnote{In practice one uses accumulating remainder \emph{forests} for improved efficiency, especially with regard to memory usage. As we are using the technique as a black box, we ignore this distinction.} to compute the products and remainders in essentially linear time.

Since the expression we use for hypergeometric traces involves the $p$-adic Gamma function $\Gamma_p$, we first describe average polynomial time algorithms to compute this function. These reduce to applications of remainder trees quite close to the original work of
Costa--Gerbicz--Harvey \cite{costa-gerbicz-harvey-14} that inspired Harvey's work on $L$-functions, which may be of independent interest.

Our main algorithm uses a variant of Harvey's \emph{generic prime} construction \cite[\S 4.4]{harvey-15}.
In its simplest form, this consists of using matrices over the truncated polynomial ring $\Z[P]/(P^e)$, arranged in such a way that truncating the matrix product at an index depending on $p$ and specializing along the
map $\Z[P]/(P^e) \to \Z/(p^e)$ taking $P$ to $p$ yields the desired result. In practice, such a product can be handled by encoding it as a product of block matrices over $\Z$; see \cref{example:generic prime}.

Here, we adapt the construction of \cite{costa-kedlaya-roe20} to use a block lower triangular matrix to record both a product over $\Z[P]/(P^e)$ and the sum of the partial products.
Two new complications arise for $e>1$: we must perform an additional transformation on the product before adding it to the sum, and we must also incorporate the contribution of certain structural constants depending on $p$, which we treat as further unknowns. (These unknowns are derived from the series expansions for $\Gamma_p$ and from the difference between the evaluation point $z$ and the $(p-1)$-st root of unity in its mod-$p$ residue disc.)

As in \cite{costa-kedlaya-roe20}, we have implemented the algorithm described above (for arbitrary $e$) in \Sage{} \cite{sage} with some low-level code written in Cython, including a wrapper to Sutherland's C library rforest (wrapped in Cython) to implement \cref{thm:remainder tree}.
Some sample timings are included in \cref{table:timings}, including comparisons with \Sage{} and \Magma{} \cite{magma}; see \cref{sec:timings} for explanation.
These confirm that the quasilinear complexity shows up in practice, not just asymptotically.  Our code, which also includes the algorithm from \cite{costa-kedlaya-roe20} for $e=1$, is available on GitHub \cite{amortizedHGMrepo} at

\begin{center}
\href{https://github.com/edgarcosta/amortizedHGM}{\texttt{https://github.com/edgarcosta/amortizedHGM}}.
\end{center}

The broader context of our work is the desire to tabulate
hypergeometric $L$-functions at scale in the L-Functions and Modular Forms Database (LMFDB) \cite{lmfdb}, in part to investigate the \emph{murmurations} phenomenon for these $L$-functions
\cite{murmurations1, murmurations2}.
Our timings suggest that this prospect is within reach; we discuss this in \cref{sec:tabulation}. In particular, while there should exist an analogue of \cref{T:main} for $p^f$-Frobenius traces for any fixed $f>1$, it does not seem to be needed in practice.

\section{Accumulating remainder trees and generic primes}

We use accumulating remainder trees to amortize the computation of the trace formula, following \cite{costa-kedlaya-roe20}. As we use this construction as a black box, we recall only the structure of the input and output and the overall complexity estimates.

\begin{theorem} \label{thm:remainder tree}
Fix a positive integer $e$. Suppose we are given
\begin{itemize}
\item
a list of $r \times r$ matrices $A_0,\dots,A_{b-1}$ over $\Z$,
\item
a list of primes $p_1,\dots,p_c$, and
\item
a list of distinct cut points $0 \leq b_1,\dots,b_c \leq b$.
\end{itemize}
Let $B$ be an upper bound on the bit size of $\prod_{j=1}^c p_j$ and $H$ an upper bound
on the bit size of any $p_i^e$ or any entry of $A_i$.
Assume also that $\log r = O(H)$ and $r = O(\log b)$.
Then there is an algorithm that computes
\[
C_n \colonequals A_0 \cdots A_{b_n-1} \bmod{p_n^e} \qquad (1 \le n < c)
\]
with time complexity
\[
O(r^2(eB + bH)\log(eB + bH)\log b)
\]
and space complexity
\[
O(r^2(eB + bH)\log b).
\]
\end{theorem}
\begin{proof}
This follows from \cite[Thm. 3.2]{harvey-sutherland-16}\footnote{The complexities in \textit{loc. cit.} are stated in terms of the complexity of multiplying two $n$-bit integers. Per \cite{harvey-vanderhoeven19} we take this to be $O(n \log n)$.} (an improvement of \cite[Thm. 4.1]{harvey-sutherland-14}) via a change of notation as in \cite[Definition~3.1, Algorithm~2]{costa-kedlaya-roe20}. \end{proof}

\begin{example} \label{example:paradigm}
In practice, we will apply \cref{thm:remainder tree} in a restricted fashion.
\begin{itemize}
\item
The matrix $A_i$ will be the specialization of a single matrix $A$ over $\Z[k]$ at $k=i$, whose entries have degree $O(d)$ and coefficients of bit size $O(d \log X)$.
\item
The primes $p_i$ will all lie in a fixed arithmetic progression bounded by $X$.
\item
The cut point $b_i$ will be a linear function of $p_i$ (up to rounding) whose coefficients are $O(1)$.
\end{itemize}
We may then take $b = O(X)$, $B = O(X)$, and $H = O((d+e) \log X)$.
In particular, $eB+bH = O((d+e) X \log X)$; assuming that $d+e = O(X)$ and $r = O(\log X)$,
the time and space complexities in \cref{thm:remainder tree} become
respectively
\[
O((d+e)r^2 X (\log X)^3) \qquad
\mbox{ and } \qquad O((d+e)r^2 X (\log X)^2).
\]
\end{example}

\begin{example}
One basic instantiation of \cref{example:paradigm} is to batch-compute the quantities $(\lceil \gamma p \rceil - 1)!  \pmod{p^e}$
for some $\gamma \in (0,1] \cap \Q$. For instance, the case $e=2,\gamma=1$ is the focus of \cite{costa-gerbicz-harvey-14}.
\end{example}

\begin{example} \label{example:harmonic sums}
    Let $j$ be a positive integer and choose $\gamma \in (0,1]$.
    We may also use \cref{example:paradigm} to batch-compute the quantities
\begin{equation} \label{eq:Hjgamma}
H_{j,\gamma}(p) \pmod{p^e}, \qquad H_{j,\gamma}(p) \colonequals \sum_{i=1}^{\lceil \gamma p \rceil - 1} i^{-j}
\end{equation}
by interpreting them as
\[
H_{j,\gamma}(p) = S(p)_{21}/S(p)_{11}, \qquad
S(p) \colonequals \prod_{i=1}^{\lceil \gamma p \rceil - 1} \begin{pmatrix} i^j & 0 \\
1 & i^j \end{pmatrix}.
\]
\end{example}

\begin{remark} \label{R:select row}
In \cref{example:harmonic sums}, we can also write
\[
H_{j,\gamma}(p) = (VS(p))_{11}/(VS(p))_{12}, \qquad
V \colonequals \begin{pmatrix} 0 & 1 \end{pmatrix}.
\]
Prepending $V$ to the product yields some significant computational savings, by reducing the size of the intermediate products in the remainder tree computation.
\end{remark}

\begin{notation}\label{notation:coeff_notation}
We write $f^{[h]}$ for the coefficient of $x^h$ in the polynomial $f(x)$.
\end{notation}

\begin{example} \label{example:generic prime}
The paradigm of \cref{example:paradigm} excludes the possibility of computing expressions involving $p$ other than via the cut point.
Harvey's ``generic prime'' construction circumvents this issue by instead computing over $\Z[x]/(x^e)$, where $x$ is specialized to $p$ as a postprocessing step. We use a variant of this idea in \cref{sec:our algorithm}.

In lieu of implementing \cref{thm:remainder tree} with $\Z$ replaced by $\Z[x]/(x^e)$, we encode a matrix product over $\Z[x]/(x^e)$ via a lower triangular block representation
\begin{equation}
f \mapsto (f^{[h_1-h_2]})_{h_1,h_2=1,\dots,e}.
\end{equation}
One could alternatively represent polynomials via Sch\"onhage's method, i.e., giving their evaluations at a large power of 2 \cite{schonhage}.
In theory, this would improve our runtime by replacing matrix multiplication with integer multiplication.
However, in \cref{sec:our algorithm} we will manipulate triangular matrices in a way that seems incompatible with this strategy.
\end{example}

\begin{remark} \label{R:hypergeometric product}
\cref{example:harmonic sums} can be interpreted either as computing $\prod_{i=1}^{\lceil \gamma p \rceil-1} (i^j + x) \pmod{(x^2, p^e)}$ as in \cref{example:generic prime},
or as a ``hypergeometric'' construction using
\[
\prod_{i=1}^{\lceil \gamma p\rceil - 1} \begin{pmatrix}
g(i-1) & 0 \\
1 & f(i)
\end{pmatrix}
\mbox{ to compute } \sum_{i=1}^{\lceil \gamma p\rceil - 1} \prod_{j=1}^{i-1} \frac{f(j)}{g(j)},
\]
taking $f(i) \colonequals i^j, g(i) \colonequals (i+1)^j$. The latter form also covers the construction in \cite{costa-kedlaya-roe20} and inspires, but does not directly encompass, the construction in \cref{subsec:form of product}.
\end{remark}

\begin{remark} \label{R:projective}
In many cases, \cref{thm:remainder tree} is used in a ``projective'' manner, in that the matrix product is only needed up to scalar multiples
(e.g., all cases covered by \cref{R:hypergeometric product}, and the construction of  \cref{subsec:form of product}). It may be possible to exploit this to reduce the complexity of some intermediate matrices, but we did not pursue this.

To illustrate the potential savings, let us make these common factors explicit in \cref{example:harmonic sums}. In the ring $ \Z[x]/(x^2)$, $\prod_{i=1}^k (i^j + x)$ is divisible by $\prod_{i=1}^{k} \tilde{f}(i)$ where
\[
\tilde{f}(i) = \begin{cases} (i/p)^j & \mbox{if $i = p^e$ for some prime $p$ and some $e>0$,} \\
i^j & \mbox{otherwise}.
\end{cases}
\]
\end{remark}

\section{The \texorpdfstring{$p$}{p}-adic Gamma function}

We next recall some properties of the $p$-adic Gamma function, and then give an average polynomial time
algorithm to compute series expansions of it at rational arguments.
This may be of independent interest.

To avoid some complications, notably in \eqref{eq:Lipschitz continuity}, we  assume $p>2$.

\begin{definition}
\label{def:Gamma_p}
The (Morita) $p$-adic Gamma function is the unique continuous function $\Gamma_p\colon \Z_p \to \Z_p^\times $ which satisfies
\begin{equation}
\label{equation:padicgamma-integers}
    \Gamma_p(n+1) = (-1)^{n+1} \prod_{\substack{i=1 \\ (i,p) = 1}} ^n i  = (-1)^{n+1} \frac{\Gamma(n+1)}{p^{\lfloor n/p \rfloor} \Gamma(\lfloor n/p \rfloor + 1)}
\end{equation}
for all $n \in \Z_{\geq 0}$.
It satisfies the functional equations
\begin{gather}\label{eq:gamma_feq}
    \Gamma_p (x + 1) = \omega(x)\Gamma_p (x),
    \qquad
    \omega(x) \colonequals
    \begin{cases} -x & \text{if } x \in \Z_p^\times, \\
    -1 & \text{if } x \in p\Z_p,
    \end{cases} \\
    \label{eq:reflection identity}
\Gamma_p(x) \Gamma_p(1-x) = (-1)^{x_0},
\end{gather}
in which $x_0 \in \{1,\dots,p\}$ is congruent to $x \in \Z_p$ mod $p$. There is also an analogue of the Gauss multiplication formula \cite[\S VII.1.3]{robert-98}; it may be possible to use this to streamline our algorithms, but we have not pursued this.
\end{definition}

\begin{definition}
As in \cite[\S VII.1.6]{robert-98} or~\cite[\S 6.2]{rodriquez-villegas-07}, for any $a \in \Z_p$ the restriction of $\Gamma_p$ to the disc $a + p\Z_p$ admits a series expansion
\begin{equation} \label{eq:Gammap series}
\Gamma_p(x+a) = \sum_{i=0}^\infty c_i x^i \qquad (c_i \in \Z_p)
\end{equation}
 In particular, $\Gamma_p$ is Lipschitz continuous with $C = 1$, i.e.,
\begin{equation} \label{eq:Lipschitz continuity}
\left| \Gamma_p(x) - \Gamma_p(y)  \right|_p \leq \left| x - y\right|_{p}.
\end{equation}
\end{definition}

\begin{definition} \label{D:log gamma}
Let $\log\colon \Z_p^\times \to \Z_p$ denote the $p$-adic logarithm, which vanishes on roots of unity and is given by the usual power series on $1 + p\Z_p$.
Then \eqref{eq:Gammap series}  immediately implies that $\log \Gamma_p$ admits a series expansion around any $a \in \Z_p$
with coefficients in $\Z_p$.

For $\gamma \in (0,1]$ and $x \in p\Zp$, with $H_{j, \gamma}(p)$ as defined in \eqref{eq:Hjgamma}, \eqref{eq:gamma_feq} implies that
\begin{equation}
\begin{aligned} \label{eq:Gammap translate by a}
\log \frac{\Gamma_p(x+\lceil \gamma p \rceil)}{\Gamma_p(\lceil \gamma p \rceil)}
&= \log \frac{\Gamma_p(x)}{(\lceil \gamma p \rceil - 1)!}\prod_{i = 0}^{\lceil \gamma p \rceil-1} \frac{\Gamma_p(x + i +1)}{\Gamma_p(x+i)}
\\
&= \log \Gamma_p(x) +
\sum_{i=1}^{\lceil \gamma p \rceil-1} \log \left(1 + \frac{x}{i}\right) \\
&= \log \Gamma_p(x) -
\sum_{j=1}^{\infty} \frac{(-x)^j}{j} H_{j,\gamma}(p).
\end{aligned}
\end{equation}
\end{definition}

\begin{remark} \label{R:quadratic time}
We will use a na\"\i ve quadratic estimate for the runtime of computing the $p$-adic logarithm and exponential; while asymptotically these require quasilinear time \cite[\S 7,8]{bernstein-08}, \cite[\S 3]{caruso}, we will consider input sizes much smaller than the asymptotic crossover and so the quadratic estimates are more accurate in practice. Similar considerations will apply to integer multiplication outside of \cref{thm:remainder tree}.
\end{remark}

We first give an amortized computation of the expansion of $\Gamma_p$ around $0$.
It is equivalent, and will be convenient for later, to work instead with $\log \Gamma_p$.

\begin{algorithm}[ht] \label{alg:dwork-mahler expansions1}
\caption{Expansion of $\log \Gamma_p(x)$ modulo $p^e$ (\cref{T:Gamma expansions1})}
Apply \cref{thm:remainder tree} to compute $(p-1)! \pmod{p^e}$ for all $p \in T_X$.

For $j=1,\dots,e-2$, apply \cref{thm:remainder tree} as in \cref{example:harmonic sums} to compute $H_{j,1}(p) = \sum_{i=1}^{p-1} i^{-j} \pmod{p^{e-j}}$ for all $p \in T_X$.

Let $A$ be the upper triangular $(e-1) \times (e-1)$ matrix with $A_{ij} = \binom{j}{i-1}$ for $1 \leq i \leq j \leq e-1$.
For each $p \in T_X$, form the vectors $v,w \in (\Z/p^e\Z)^{e-1}$ given by
\[
v_j = \begin{cases} \log (-(p-1)!) & \mbox{if $j=1$} \\
\tfrac{(-1)^j}{j-1} p^{j-1} H_{j-1,1}(p) & \mbox{if $j>1$},
\end{cases} \qquad w = A^{-1} v.
\]
Then $\log \Gamma_p(py) \equiv \sum_{j=1}^{e-1} w_j y^j \pmod{p^e}$ in $\Z[y]$.
\end{algorithm}

\begin{theorem} \label{T:Gamma expansions1}
Let $T$ be the set of primes.
Then Algorithm~\ref{alg:dwork-mahler expansions1}  computes
the series expansions of $\log \Gamma_p(p y)$ in $\Z[y]/(p^e)$ for all primes $p \in T_X \colonequals T \cap [e+1, X]$,
with respective time and space complexities
\[
O(e^2 X (\log X)^3 + (e^4 \log e)X) \qquad
\mbox{ and } \qquad O(e X (\log X)^2 + e^2 X).
\]
\end{theorem}

\begin{proof}
To see that the output is correct, fix a prime $p \geq e+1$.
For $j \not\equiv 0 \pmod{p-1}$, $H_{j,1}(p) \equiv 0 \pmod{p}$ and so $(p y)^j H_{j,1}(p) \equiv 0 \pmod{p^{j+1}}$. Hence the terms $j > e-2$ do not contribute modulo $p^e$ in \eqref{eq:Gammap translate by a}, so we may rewrite the latter as
\begin{equation*}
    \begin{aligned}
        \log \Gamma_p(p(y + 1)) - \log \Gamma_p(py)
                                                    &\equiv \log \bigl({ -(p-1)!}\bigr) - \sum_{j=1}^{e-2} \frac{(-p y)^j}{j} H_{j,1}(p) \pmod{p^{e}}\\
                                                &\equiv \sum_{j=1}^{e-1} v_j y^{j-1} \pmod{p^{e}}.
    \end{aligned}
\end{equation*}
Therefore, the right hand side is the image of $\log \Gamma_p(py)$ under the difference operator, which is represented by $A$. Since $A$ has diagonal entries $1,\dots,e-1$, it is invertible over $\Z_p$.
Since $\Gamma_p(0) = 1$,
$A^{-1} v$ contains the coefficients of the expansion of $\log \Gamma_p(py)$ in $y$.

For the complexity estimate, we may assume $e < X$ as otherwise there is nothing to do. We cover Steps 1 and 2 by applying \cref{thm:remainder tree} $e-1$ times as in \cref{example:paradigm} with a degree bound of $O(e)$.
This costs $O(e^2 X (\log X)^3)$ time and $O(eX(\log X)^2)$ space, plus $O(e^2 X)$ space to record the results.
In Step 3, for each of $O(X/\log X)$ primes $p$, we perform
one logarithm and $O(e^2)$ multiplications of an integer in $\Z/(p^e)$ with an entry of $A^{-1}$ of bitsize\footnote{This estimate follows by \href{https://oeis.org/A080779}{expressing} $A^{-1}$ in terms of Bernoulli numbers.} $O(e \log e)$; as per \cref{R:quadratic time} each multiplication costs $O(e^2 \log e \log X)$ time.
\end{proof}

\begin{remark}
One can extend Algorithm~\ref{alg:dwork-mahler expansions1} to
cover $p > \tfrac{e}{2}$ using the fact that $\log \Gamma_p$ is an odd series \cite[\S VII.1.5, Theorem]{robert-98}.
For smaller $p$, we can replace the use of interpolation in step 2 with
a direct application of \eqref{eq:Gammap translate by a}
to solve for the coefficients of $\Gamma_p(x)$. As we will assume $p>e$ later, we did not implement these steps.
\end{remark}

We next expand around other values $\gamma \in (0,1) \cap \Q$. For these values, it is more useful to retain neither $\Gamma_p(py+\gamma)$ nor its logarithm, but something in between.

\begin{algorithm}[ht] \label{alg:dwork-mahler expansions3}
\caption{Expansion of $\log \Gamma_p(x+\gamma)$ modulo $p^e$ (\cref{T:Gamma expansions})}
Use Algorithm~\ref{alg:dwork-mahler expansions1}
to compute $\log \Gamma_p(py) \pmod{p^e}$ for all $p \in T_X$.

For each $\gamma \in S_d$, use \cref{thm:remainder tree} to compute $\Gamma_p(\lceil \gamma p \rceil) = \pm (\lceil \gamma p\rceil -1)! \pmod{p^e}$ for all $p \in T_X$.

For each $\gamma \in S_d$, for $j=1,\dots,e-1$, apply \cref{thm:remainder tree} as in \cref{example:harmonic sums} to compute $H_{j,\gamma}(p) \pmod{p^{e-j}}$ for all $p \in T_X$.

For each $p \in T_X$ and $\gamma \in S_d$, use Step 1 and \eqref{eq:Gammap translate by a} to compute $
\log \tfrac{\Gamma_p(py+\lceil \gamma p \rceil)}{\Gamma_p(\lceil \gamma p \rceil)}$ as an element of $\Z[y]/(p^e)$.

For each $p \in T_X$ and $\gamma \in S_d$,
set $b \colonequals - (\gamma d)/ p \pmod{d}$.
Set $c_{\gamma,p} \colonequals \Gamma_p\bigl(\lceil \tfrac{b}{d} p \rceil\bigr)$ and $s_{\gamma,p}(py) \colonequals \log \Gamma_p\bigl(py+\lceil \tfrac{b}{d} p \rceil\bigr)/\Gamma_p\bigl(\lceil \tfrac{b}{d} p \rceil)\bigr|_{y=y-\tfrac{b}{d}}$.
\end{algorithm}

\begin{theorem} \label{T:Gamma expansions}
Fix integers $e,d \geq 2$.
Let $S_d$ be the set of $\gamma \in \Q \cap (0,1)$ of the form $\frac{c}{d}$ with $\gcd(c,d)=1$.
Let $T$ be the set of primes not dividing $d$.
Then Algorithm~\ref{alg:dwork-mahler expansions3} computes,
for all $\gamma \in S_d$ and all $p \in T_X \colonequals T \cap [e+1, X]$,
some quantities
\begin{itemize}
    \item $c_{\gamma,p} \in \Z/(p^e)$ and
    \item $s_{\gamma,p}(py) \in \Z[y]/(p^e)$
\end{itemize}
such that
\begin{equation*}
    \Gamma_p(py + \gamma) \equiv c_{\gamma,p} \exp s_{\gamma,p}(py) \pmod{p^e},
\end{equation*}
with respective time and space complexities bounded by
\[
O(de^2 X (\log X)^3 + de^4 (\log d + \log e) X) \qquad
\mbox{ and } \qquad O(e X (\log X)^2 + de^2 X).
\]
\end{theorem}
\begin{proof}
To see that the output is correct, we first show that $p y + \gamma =  p (y - b/d) +  \lceil \tfrac{b}{d} p \rceil$.
    Set $a \colonequals \lceil \tfrac{b}{d} p \rceil$, $a' \colonequals \gamma \pmod{p}$,
and define $b' \in \Z$ by $\gamma d = a'd - b'p$; then
$\tfrac{b'}{d} \in S_d$, $a' = \tfrac{b'}{d} p + \gamma = \lceil \tfrac{b'}{d} p \rceil$, and $\gamma d\equiv -b'p \pmod{d}$.
We deduce that $a=a', b=b'$ and hence $a \equiv \gamma \pmod{p}$, and $b = d(a-\gamma)/p$.
We thus reduce to \eqref{eq:Gammap translate by a} and \cref{T:Gamma expansions1}.

For the complexity estimate, we may again assume $e < X$. We cover Step 1 using \cref{T:Gamma expansions1}. We cover Steps 2 and 3 by applying \cref{thm:remainder tree} $O(de)$ times  as in \cref{example:paradigm} with a degree bound of $O(e)$. This costs $O(de^2 X (\log X)^3)$ time and
$O(e X (\log X)^2)$ space, plus $O(de^2 X)$ space to record the results; this cost also covers Step 4.
In Step 5, for each of $O(d)$ values of $b$ and $O(X/\log X)$ primes $p$, we perform a polynomial substitution using $O(e^2)$
multiplications of an element of $\Z/(p^e)$ by a rational of bitsize $O(e \log d)$;
as per \cref{R:quadratic time} each multiplication costs $O(e^2 \log d \log X)$ time.
\end{proof}

\begin{remark}
In practice, we obtain a speedup by a factor of $2$ in Algorithm~\ref{alg:dwork-mahler expansions3} by working over $S_d \cap (0,\tfrac{1}{2}]$ in Steps 1--4. In Step 5, if $b > \tfrac{d}{2}$,
in light of \eqref{eq:reflection identity}
we may set $c_{\gamma,p} \colonequals (-1)^{\lceil \gamma p \rceil} c_{1-\gamma,p}^{-1}$ and $s_{\gamma,p}(x) \colonequals -s_{1-\gamma,p}(-x)$.
\end{remark}

\begin{remark}
Note that computing $H_{1,\gamma}(p) \pmod{p^e}$ via \cref{example:harmonic sums} yields
$(\lceil \gamma p \rceil - 1)! \pmod{p^e}$ as a byproduct. In practice, in Step 2 of  Algorithm~\ref{alg:dwork-mahler expansions1} we compute $H_{1,1}$ mod $p^e$ rather than $p^{e-1}$, and we skip Step 1.
We similarly modify Steps 2 and 3 of  Algorithm~\ref{alg:dwork-mahler expansions3} when $e>1$.
\end{remark}

\section{The Beukers--Cohen--Mellit trace formula}

We summarize \cite[\S 2.2]{costa-kedlaya-roe20} primarily for the purpose of setting notation.

\begin{definition}
    \label{def:datum}
A \emph{hypergeometric datum} is a pair of disjoint tuples $\alpha=(\alpha_1,\ldots,\alpha_r)$ and $\beta=(\beta_1,\ldots,\beta_r)$
valued in $\Q \cap [0,1)$. Such a pair is \emph{Galois-stable} (or \emph{balanced}) if any two reduced fractions with the same denominator occur with the same multiplicity.

For the rest of the paper, fix a Galois-stable\footnote{Without the Galois-stable condition, much of this discussion carries over, but the resulting motives are defined not over $\Q$ but some cyclotomic field.} hypergeometric datum $\alpha,\beta$ and some $z \in \Q \setminus \{0,1\}$.
We say that a prime $p$ is \emph{wild} if it divides the denominator of some $\alpha_j$ or $\beta_j$; \emph{tame} if it is not wild but divides the numerator or denominator of $z$ or the numerator of $z-1$; and \emph{good} otherwise.
\end{definition}

\begin{definition}\label{def:zigzag function}
The \emph{zigzag function} $Z_{\alpha, \beta} : [0, 1] \to \Z$ is defined by
\[
Z_{\alpha, \beta}(x) \colonequals \#\{j\colon \alpha_j \leq x\} - \#\{j\colon \beta_j \leq x\}.
\]
In terms of the zigzag function, the \emph{minimal motivic weight} is given by
\begin{equation}\label{eq:mot_wt}
\begin{split}
w &= \max\{Z_{\alpha,\beta}(x)\colon x \in [0,1]\} -
\min\{Z_{\alpha,\beta}(x)\colon x \in [0,1]\} - 1 \\
&=  \max\{Z_{\alpha,\beta}(x)\colon x \in \alpha\} -
\min\{Z_{\alpha,\beta}(x)\colon x \in \beta\} - 1.
\end{split}
\end{equation}
Write $\{x\} \colonequals x - \lfloor x \rfloor$ for the fractional part of $x \in \Q$ and $\#S$ for the cardinality of a set $S$.  Set
\begin{align}
\label{eq:eta_m}
\eta_m(x_1,\ldots,x_r) &\colonequals\sum_{j=1}^r\left(\left\{x_j+\tfrac{m}{1-p}\right\}-\left\{x_j\right\}\right), \\
\label{eq:xi}\xi_m(\beta) &\colonequals\#\{j:\beta_j=0\}-\#\left\{j:\beta_j+\tfrac{m}{1-p}=0\right\},  \\
\label{eq:D}
D &\colonequals \tfrac{1}{2}(w + 1 - \#\{j: \beta_j=0\}).
\end{align}
\end{definition}

\begin{definition} \label{D:hypergeometric trace}
For $p$ prime,
define a $p$-adic analogue of the Pochhammer symbol by setting
\begin{equation}\label{eq:poch_def}
(x)_{m}^* \colonequals \frac{\Gamma_p\left(\left\{ x+\tfrac{m}{1-p}\right\}\right) }{\Gamma_p(\{x\})}.
\end{equation}
Let $[z] \in \Z_p$ be the unique $(p-1)$-st root of unity congruent to $z$ modulo $p$; note that this should not be confused with the notation $f^{[h]}$ defined in \cref{notation:coeff_notation}.
As in~\cite[\S~2]{watkins-15}, for $p$ good we write
\begin{equation}\label{eq:Hq}
\Hpfull \colonequals \frac{1}{1-p}\sum_{m=0}^{p-2}(-p)^{\eta_m(\alpha)-\eta_m(\beta)}p^{D + \xi_m(\beta)} \left(\gammaprod (\gamma)_m^* \right)[z]^m,
\end{equation}
where $\gammaprod f(\gamma)$ is shorthand for $\prod_{j=1}^r \tfrac{f(\alpha_j)}{f(\beta_j)}$.
\end{definition}

By combining \cite[Theorem~1.5]{beukers-cohen-mellit-15} (an adaptation of \cite[\S 8.2]{katz-90}) with the Gross--Koblitz formula \cite[\S VII.2.6]{robert-98} as described in \cite{watkins-15}, one establishes the following.

\begin{theorem}\label{thm:trace of Frobenius}
Assume\footnote{This point was neglected in  \cite{costa-kedlaya-roe20}.} that $0 \notin \alpha$. Then there exists a motive $M^{\alpha,\beta}_z$ over $\Q$,
pure of weight $w$ and dimension $r$,
such that  for each good prime $p$, $M^{\alpha,\beta}_z$ has good reduction at $p$ and
\[
\Hpfull = \tr(\frob | M^{\alpha, \beta} _{z} ) \in \Z \cap [-r p^{w/2},rp^{w/2}].
\]
\end{theorem}

\begin{remark} \label{rmk:sufficient precision}
The bound on $\Hp$ implies that
for $p > 4r^2$, $\Hp$ is determined by its reduction modulo $p^e$ for $e = \lceil (w+1)/2 \rceil$.
\end{remark}

\begin{remark} \label{rmk:swap alpha beta}
The definition of $M^{\alpha,\beta}_z$ does not itself require $0 \notin \alpha$, only the validity of the trace formula as written. In general, there is an isomorphism $M^{\alpha,\beta}_z \cong M^{\beta,\alpha}_{1/z}$, which in the case $0 \notin \alpha,\beta$
corresponds to a symmetry in \eqref{eq:Hq}: the substitution $[z] \mapsto [1/z]$ carries the summand indexed by $m$ to the summand indexed by $p-1-m$.
The term-by-term equality can be seen from \eqref{eq:gamma_feq}, taking care about signs.

When $0 \in \alpha$, we currently compute Frobenius traces by applying \cref{thm:trace of Frobenius} to $M^{\beta,\alpha}_{1/z}$. It would be preferable to adapt \eqref{eq:Hq} to handle this case directly, so as to free up the swap of $\alpha$ and $\beta$ for other uses (see \cref{R:clear denominators}).
\end{remark}

\section{Computation of hypergeometric traces}
\label{sec:our algorithm}

We next exhibit an algorithm (Algorithm~\ref{alg:Hp overall}) that, on input of $\alpha,\beta,X,e$, computes
$\Hp \pmod{p^e}$ for all primes $p \leq X$ excluding tame\footnote{We can also handle tame primes where $z \in \Z_p^\times$. In particular, we can take $z=1$.} and wild primes.
The complexity analysis of this algorithm will yield \cref{T:main}; see \cref{subsec:complexity}.

It is harmless to also exclude a finite set of small good primes, as they can be handled easily by directly computing \eqref{eq:Hq} modulo a suitable power of $p$.
We will restrict attention to $p$ with
\begin{equation} \label{eq:lower bound on p}
 \max\{e, d(d-1)\} < p \leq X,
\end{equation}
where $d$ is the maximum of the denominators of $\alpha \cup \beta$.

To help navigate some heavy notation, we summarize it in the following table.

\begin{center}
\begin{tabular}{cccccc}
Symbol & Reference & Symbol & Reference & Symbol & Reference \\
\hline
$f^{[h]}$ & \cref{notation:coeff_notation}  & $\Gamma_p(x)$ & Def.~\ref{def:Gamma_p} & $\omega(x)$ & \eqref{eq:gamma_feq} \\
$\alpha,\beta,r,z$ & Def.~\ref{def:datum}
& $\{x\}, w, \eta_m, \xi_m, D$ & Def.~\ref{def:zigzag function} &
$H_p$, $[z]$, $\gammaprod$ & Def.~\ref{D:hypergeometric trace} \\
\hline
$a_i, b_i, r_i, c$ & \eqref{eq:ai bi ri} &
$f_{i,c}(k), g_{i,c}(k)$ & \eqref{eq:define fick} &
$P_m, P'_m$ & \eqref{eq:Pm} \\
$A_{i,c}(k)$ & \eqref{eq:matrix formula}
& $\gamma_i, m_i$ & \eqref{eq:distinct breaks} & $Q_{h_1,h_2}(k)$ & \eqref{eq:matrix entries} \\
$c_{i,h}(p)$ & \eqref{eq:precomputed quantities} & $\gamma_{i,c}$ & \eqref{eq:gamma ic}  & $R_{i}(x)$ & \eqref{eq:Rikx}  \\
$\delta_{h_1,h_2}$ & \eqref{eq:matrix formula} &
$h_c(\gamma, \gamma_i)$ & \eqref{eq:hc gamma} &$\sigma_i, \tau_i$ & \eqref{eq:sigma tau}\\
$\epsilon_c(\gamma, \gamma_i)$ & \eqref{eq:epsilon c} & $\iota(x,y)$ & \eqref{eq:iota xy}& $S_{i}(p)$ & \eqref{eq:tilde Sip}\\
$e_i, e'_i, \overline{\sigma}_i, \overline{\tau}_i$ & \eqref{eq:separate sigma} & $k,m$ & \eqref{eq:floor expression} \\
\end{tabular}
\end{center}

\subsection{Breaking the sum into ranges}
We start with a high-level breakdown of the algorithm, in which $e$ plays only a minor role.
For $m=0,\dots,p-2$, define
\begin{equation} \label{eq:Pm}
P_m \colonequals [z]^m \gammaprod (\gamma)^*_m, \qquad
P'_m \colonequals (-p)^{\eta_m(\alpha)-\eta_m(\beta)} p^{D + \xi_m(\beta)} P_m.
\end{equation}
Let $0 = \gamma_0 < \cdots < \gamma_s = 1$ be the distinct elements in $\alpha \cup \beta \cup \{0,1\}$.
Write $m_i$ for $\lfloor \gamma_i(p-1) \rfloor$;
by \eqref{eq:lower bound on p}, we also have
\begin{equation} \label{eq:distinct breaks}
0 = m_0 < \cdots < m_s = p-1.
\end{equation}
By \cite[Lemma~4.2]{costa-kedlaya-roe20},
there exist integers $\sigma_i, \tau_i$ for $0 \leq i \leq s-1$ such that
\begin{equation} \label{eq:sigma tau}
\frac{P'_m}{P_m} = \begin{cases} \tau_i & \mbox{if $m = m_i$,} \\
\sigma_i & \mbox{if $m_i < m < m_{i+1}$};
\end{cases}
\end{equation}
in fact we can choose $e_i, e_i' \in \{1,\dots,e\}$
and $\overline{\sigma}_i, \overline{\tau}_i \in \{-1,0,1\}$ such that for all $p$,
\begin{equation} \label{eq:separate sigma}
\sigma_i \equiv p^{e-e_i} \overline{\sigma}_i \pmod{p^e}, \qquad \tau_i \equiv p^{e-e'_i} \overline{\tau}_i \pmod{p^e}.
\end{equation}
In this notation, we summarize the method in Algorithm~\ref{alg:Hp overall}.

\begin{algorithm} \label{alg:Hp overall}
\caption{Computation of $\Hp$ for good $p$ satisfying \eqref{eq:lower bound on p}}
Apply Algorithm~\ref{alg:dwork-mahler expansions3} to each $d$ dividing the common denominator of $\gamma_i$ and $\gamma_j$ for some $i,j \in \{0,\dots,s\}$ (not necessarily distinct).

For each $i \in \{0,\dots,s-1\}$ for which $\overline{\tau}_i \neq 0$, for each $p$, compute $\overline{\tau}_i P_{m_i} \pmod{p^{e'_i}}$ as indicated in \cref{subsec:compute Pmi}.

For each $i \in \{0,\dots,s-1\}$ for which $\overline{\sigma}_i \neq 0$, for each $p$, compute the quantities $c_{i,h}(p) \pmod{p^{e_i-h}}$ as indicated in \eqref{eq:precomputed quantities}.

For each $i \in \{0,\dots,s-1\}$ for which $\overline{\sigma}_i \neq 0$, compute $\overline{\sigma}_i \sum_{m=m_i+1}^{m_{i+1}-1} P_m \pmod{p^{e_i}}$ in terms of the quantities $c_{i,h}(p)$ as indicated in \cref{subsec:form of product}.

For each $p$, compute $\Hp \pmod{p^e}$ by rewriting \eqref{eq:Hq} in the form
\begin{equation}
\Hpfull \equiv \frac{1}{1-p} \sum_{i=0}^{s-1} \left( p^{e-e'_i} \overline{\tau}_i P_{m_i} + p^{e-e_i} \overline{\sigma}_i \sum_{m=m_i+1}^{m_{i+1}-1} P_{m} \right) \pmod{p^e}.
\end{equation}
\end{algorithm}

\subsection{Residue classes}
\label{subsec:residue classes}

We next separate primes into residue classes modulo the denominator of $\gamma_i$.
Following \cite[Lemma~4.1]{costa-kedlaya-roe20},
define
\begin{equation} \label{eq:iota xy}
  \iota(x, y) \colonequals \begin{cases} 1 & \mbox{if $x \le y$} \\ 0 & \mbox{if $x > y$.} \end{cases}
\end{equation}
Write $\gamma_i = \frac{a_i}{b_i}$ and define an integer $r_i \in \{0, \dots,b_i-1\}$ by
\begin{equation} \label{eq:ai bi ri}
a_i(p-1) = m_i b_i + r_i;
\end{equation}
for $p \equiv c \pmod{b_i}$ with $c \in (\Z/b_i\Z)^\times$, $r_i$ is the residue of $a_i(c-1)$ mod $b_i$. We have
\begin{equation} \label{eq:gamma ic}
m_i = \gamma_i (p-1) - \gamma_{i,c}, \qquad \gamma_{i,c} \colonequals \tfrac{r_i}{b_i}.
\end{equation}
For the remainder of \S\ref{sec:our algorithm}, we fix an index $i \in \{0,\dots,s-1\}$ and a quantity $c \in (\Z/b_i \Z)^\times$, and limit attention to primes $p \equiv c \pmod{b_i}$.

For $\gamma \in \alpha \cup \beta$, we analyze $\left\{ \gamma + \tfrac{m}{1-p} \right\}$ in terms of
\begin{equation} \label{eq:hc gamma}
h_c(\gamma, \gamma_i) \colonequals \gamma - \gamma_i + \iota(\gamma, \gamma_i) - \gamma_{i,c} \in (-1, 1].
\end{equation}
For $m = m_i + k$ with $1 \leq k \leq m_{i+1} - m_i$,
by \cite[(5.11)]{costa-kedlaya-roe20} we have
\begin{align}
\label{eq:floor expression}
\left\{\gamma + \tfrac{m}{1-p} \right\} &=
\gamma + \tfrac{m}{1-p} + \iota(\gamma, \gamma_i) \\
\nonumber
&= h_c(\gamma,\gamma_i) + (k - p\gamma_{i,c}) \tfrac{1}{1-p} \\
\nonumber
&= h_c(\gamma,\gamma_i) + k + (k - \gamma_{i,c}) \tfrac{p}{1-p}.
\end{align}
For $k=0$, we instead have
\begin{align}
\label{eq:fractional part edge case}
\left\{\gamma + \tfrac{m_i}{1-p} \right\} &= h_c(\gamma,\gamma_i)  - \tfrac{p}{1-p} \gamma_{i,c} - \epsilon_c(\gamma, \gamma_i) \\
\label{eq:epsilon c}
\mbox{ where } \epsilon_c(\gamma, \gamma_i) &\colonequals \iota(\gamma, \gamma_i) - \iota(\gamma,\gamma_i - \tfrac{1}{p-1} \gamma_{i,c}) =
\begin{cases} 1 & \mbox{if $\gamma_i = \gamma$} \\ 0 & \mbox{otherwise.}
\end{cases}
\end{align}
(We would also have $\epsilon_c(\gamma,\gamma_i)=1$ if
$\gamma_i - \frac{1}{p-1}\gamma_{i,c} < \gamma < \gamma_i$, but this would imply $\lfloor \gamma(p-1) \rfloor = \lfloor \gamma_i \rfloor$ in violation of \eqref{eq:distinct breaks}.)

We make explicit a point that was elided in
\cite[Lemma~5.10]{costa-kedlaya-roe20}.

\begin{lemma} \label{eq:not div by p}
Let $m_i \leq m < m_{i+1}$ and $\gamma \in \alpha \cup \beta$, and suppose that $\left\{\gamma + \tfrac{m}{1-p} \right\} \equiv 0 \pmod{p}$.  Then either
\begin{align*}
m &= m_i \mbox{ and } h_c(\gamma,\gamma_i) = \epsilon_c(\gamma,\gamma_i), \mbox{ or } \\
m &= m_{i+1} - 1 \mbox{ and } h_c(\gamma,\gamma_{i+1}) = \epsilon_c(\gamma,\gamma_{i+1})+1.
\end{align*}
\end{lemma}
\begin{proof}
If $m = m_i$, then by \eqref{eq:fractional part edge case}, $\left\{\gamma + \tfrac{m}{1-p} \right\} \equiv h_c(\gamma,\gamma_i) - \epsilon_c(\gamma,\gamma_i) \pmod{p}$.
Since $\epsilon_c(\gamma,\gamma_i) = 1$ implies $\gamma = \gamma_i$
and hence $h_c(\gamma,\gamma_i) = 1 - \gamma_{i,c} \geq 0$, $h_c(\gamma,\gamma_i) - \epsilon_c(\gamma,\gamma_i)$ is in $\Q \cap (-1, 1]$ with denominator at most $d(d-1)$, and so by \eqref{eq:lower bound on p} can only be divisible by $p$ if it is zero.

Now assume that $m_i < m < m_{i+1}$ and $\left\{\gamma + \tfrac{m}{1-p} \right\} \equiv 0 \pmod{p}$.
Write $\gamma +\iota(\gamma,\gamma_i) = \tfrac{a}{d}$ with $0 \leq a < 2d$ and $\gcd(a,d) = 1$.
By our hypothesis plus \eqref{eq:floor expression},
\begin{equation*}
    \frac{a}{d} = \gamma +\iota(\gamma,\gamma_i)  = \left\{\gamma + \tfrac{m}{1-p} \right\} + \frac{m}{p-1}.
\end{equation*}
Thus,
\begin{equation} \label{eq:tpdef}
a + md = (a-t)p
\end{equation}
for some $t \in \Z_{\ge 0}$, and $t \le a$ since $a + md$ is non-negative.
Furthermore, since every divisor of $a + md$ is prime to $d$, we see that $\gcd(a-t,d) = 1$,
which by \cref{def:datum} ensures that $\tfrac{a-t}{d} = \gamma_j + \delta$ for some $\delta \in \{0, 1\}$ and some $j \in \{0,\dots,s\}$.
Rewriting \eqref{eq:tpdef} in the form $(a-t)(p-1) = md+t$, we get
\begin{equation*}
    (\gamma_j + \delta) (p-1) = m + \tfrac{t}{d}.
\end{equation*}
Taking floors on both sides we obtain
\begin{equation*}
    m = \delta(p-1) + m_j - \lfloor \tfrac{t}{d} \rfloor.
\end{equation*}
Since $0 \le t \le a < 2d$, this gives a contradiction unless $t \geq d$ and either $\delta = 0$, $m = m_{i+1}-1$, and $m_j = m_{i+1}$
or $\delta = 1$, $m = m_s-1$, and $m_j = 0$; in either case, \eqref{eq:distinct breaks} implies $\gamma_j +\delta = \gamma_{i+1}$. Then \eqref{eq:tpdef} becomes
\[
\gamma + \iota(\gamma, \gamma_i) = \frac{a}{d} = \gamma_{i+1}p - m_{i+1}+1 = \gamma_{i+1} + \gamma_{i+1,c} + 1.
\]
Since $\iota(\gamma, \gamma_i) = \iota(\gamma, \gamma_{i+1}) - \epsilon_c(\gamma, \gamma_{i+1})$,
we obtain $h_c(\gamma, \gamma_{i+1}) -\epsilon_c(\gamma, \gamma_{i+1}) = 1$.
\end{proof}

\subsection{Computation of \texorpdfstring{$P_{m_i}$}{Pmi}}
\label{subsec:compute Pmi}

We next elaborate Step 2 of Algorithm~\ref{alg:Hp overall}.
 Similar ideas will also be used in Step 3; see \cref{subsec:more factorization}.

To begin with, compute $z^m \pmod{p^e}$ by repeated squaring.
Since $\log [z] = 0$, we may then obtain $[z]^m \pmod{p^e}$ by writing
\begin{equation}\label{eq:mult lift serie}
\left( \tfrac{[z]}{z} \right)^m = \exp\left(\tfrac{m}{1-p} \log z^{p-1}\right) \equiv \sum_{h=0}^{e-1} \frac{1}{h!} \left( \tfrac{1}{1-p} \log z^{p-1} \right)^h
m^h \pmod{p^e}.
\end{equation}

Next, use the output of Step~1 of Algorithm~\ref{alg:Hp overall}
to recover
\begin{equation} \label{eq:gamma base}
\gammaprod \Gamma_p(\gamma)  \pmod{p^e}
\end{equation}
and, for each $\gamma \in \alpha \cup \beta$, a constant $c_{i,\gamma,p}$ and a series $s_{i,\gamma,p}(x)$ such that for all $x \equiv 0 \pmod{p}$,
\begin{equation} \label{eq:precomputation result not normalized}
c_{i,\gamma,p} \exp s_{i,\gamma,p}(x) \equiv \Gamma_p(x+\{h_c(\gamma,\gamma_i)\}) \pmod{p^e}.
\end{equation}
Next, use \eqref{eq:gamma_feq}, \cref{eq:not div by p}, and \eqref{eq:precomputation result not normalized},
keeping in mind that $h_c(\gamma, \gamma_i)-\epsilon_c(\gamma,\gamma_i) \in (-1,1]$, to obtain an analogous representation of
\begin{equation} \label{eq:gamma product series mi}
\gammaprod \Gamma_p\left(x+h_c(\gamma, \gamma_i) - \epsilon_c(\gamma,\gamma_i) \right) \pmod{p^e}.
\end{equation}
Finally, compute $P_{m_i}$ as $z^{m_i} \pmod{p^e}$, times \eqref{eq:mult lift serie} evaluated at $m = m_i$, times \eqref{eq:gamma product series mi} evaluated at $x = \gamma_{i,c}\tfrac{p}{1-p}$, divided by \eqref{eq:gamma base}.

\begin{remark}
For the most part, in practice we make these computations modulo
$p^{e_i'}$ rather than $p^e$. The sole exception is \eqref{eq:gamma base}, which we use again in \eqref{eq:precomputed quantities}. That said, the balanced condition ensures that \eqref{eq:gamma base} is a fourth root of unity \cite[\S VII.1.3, Lemma]{robert-98}, so it suffices to compute it modulo $p$.
\end{remark}

\subsection{Factorization of the quotient}
\label{subec:factorization}

Before continuing through the remaining steps of Algorithm~\ref{alg:Hp overall}, we give a high-level description of what these steps are doing, then link this back to \cite{costa-kedlaya-roe20}.

We will define in  \cref{subsec:form of product} a block triangular matrix $A_{i,c}(k)$ over $\Z[k]$ for which
\begin{equation} \label{eq:block decomposition}
A_{i,c}(1) \cdots A_{i,c}(k) =
\begin{pmatrix}
\Delta & 0 \\
\Sigma & \Pi
\end{pmatrix},
\end{equation}
where $\Delta$ is a scalar matrix, $\Delta^{-1} \Sigma$ ``records'' $\overline{\sigma}_i \sum_{j=1}^{k} P_{m_i+j} \pmod{p^{e_i}}$ and $\Delta^{-1} \Pi$ ``records''
$P_{m_i+k+1} \pmod{p^{e_i}}$ in a sense to be specified later.
We then apply \cref{thm:remainder tree} and \cref{example:paradigm}, noting the dependence of $m_{i+1}-m_i$ on $p$, to compute
\begin{equation} \label{eq:tilde Sip}
S_{i}(p) \colonequals A_{i,c}(1) \cdots A_{i,c}(m_{i+1}-m_i-1) \pmod{p^{e_i}}
\end{equation}
for all $p$ at once, then extract the desired sum from $S_i(p)$ for each $p$ separately. (As in \cref{R:projective}, we are using the matrix product in a ``projective'' fashion.)

Let us recall how this was done for $e_i=1$ in \cite{costa-kedlaya-roe20}. For $m \colonequals m_i+k$, write
\begin{equation}
\label{eq:factorization of product}
\frac{P_{m}}{P_{m_i+1}} = [z]^{k-1} \gammaprod \frac{\Gamma_p\left(\left\{\gamma + \frac{m}{1-p}\right\}\right)}{\Gamma_p\left(\left\{\gamma + \frac{m_i+1}{1-p}\right\}\right)}
= [z]^{k-1} \gammaprod \frac{\Gamma_p\left(h_c(\gamma,\gamma_i) + k + \frac{(k-\gamma_{i,c})p}{1-p} \right)}{\Gamma_p\left(h_c(\gamma,\gamma_i) + 1 + \frac{(1 - \gamma_{i,c})p}{1-p} \right)}.
\end{equation}
As in \cite[Definition~5.7]{costa-kedlaya-roe20}, choose a positive integer $b$ to ensure that
\begin{equation}
\label{eq:define fick}
f_{i,c}(k) \colonequals b \prod_{j=1}^r \left(h_c(\alpha_j, \gamma_i) + k\right), \qquad
g_{i,c}(k) \colonequals b \prod_{j=1}^r \left(h_c(\beta_j, \gamma_i) + k\right)
\end{equation}
belong to $\Z[k]$.
By \eqref{eq:gamma_feq}, \cref{eq:not div by p},  and \eqref{eq:factorization of product},
\begin{equation}
\frac{P_{m_i+k+1}}{P_{m_i+k}} \equiv z \frac{f_{i,c}(k)}{g_{i,c}(k)} \pmod{p} \qquad (k=1,\dots,m_{i+1}-m_i-1).
\end{equation}
Write $z = \frac{z_f}{z_g}$ in lowest terms, then set
\begin{equation} \label{eq:mod p matrix}
A_{i,c}(k) \colonequals \begin{pmatrix} z_g g_{i,c}(k) & 0 \\
\overline{\sigma}_i z_g g_{i,c}(k) & z_f f_{i,c}(k) \end{pmatrix}
\equiv \mathrm{(scalar)} \begin{pmatrix} 1 & 0 \\
\overline{\sigma}_i & \tfrac{P_{m+1}}{P_{m}} \end{pmatrix} \pmod{p};
\end{equation}
we then have
\begin{equation}
\overline{\sigma}_i \sum_{m=m_i+1}^{m_{i+1}-1} P_{m} \equiv P_{m_i+1} \frac{S_i(p)_{21}}{S_i(p)_{11}} \pmod{p}.
\end{equation}

\begin{remark} \label{R:chained product}
In \cite{costa-kedlaya-roe20}, the matrices $S_i(p)$ are chained together into a single product, interleaved with connecting matrices $T_i(p)$ to account for the summands $P'_{m_i}$. We originally implemented a similar approach for $e>1$; while this saves some work at certain stages, it forces some intermediate computations to be done modulo $p^e$ rather than $p^{e_i}$, and this is disadvantageous especially with regard to memory usage. The present approach also allows more of the work to be treated as a precomputation; see \cref{sec:timings}.
\end{remark}

\subsection{More factorization of the quotient}
\label{subsec:more factorization}

To upgrade the previous discussion to handle $e_i > 1$,
we make an additional separation of factors in \eqref{eq:factorization of product},
in order to decouple the effect of shifting the argument of $\Gamma_p$ by a multiple of $p$ (accounted for by Step 3 of Algorithm~\ref{alg:Hp overall}, described below) from the effect of shifting the argument by 1 (accounted for by Step 4 of Algorithm~\ref{alg:Hp overall}, described in \cref{subsec:form of product}).

 We first observe that
in the ring $(\Frac \Z_p[k])[x]/(x^{e_i})$, we have
\begin{equation} \label{eq:Gamma argument shift}
\frac{f_{i,c}(x+k)}{g_{i,c}(x+k)}
= \gammaprod \frac{\Gamma_p(x+h_c(\gamma,\gamma_i)+k+1)}{\Gamma_p(x+h_c(\gamma,\gamma_i)+k)}.
\end{equation}
If we then define the power series
\begin{align} \label{eq:Rikx}
R_i(x) &\colonequals
\gammaprod \frac{\Gamma_p\left(x+h_c(\gamma, \gamma_i) + 1 \right)}{\Gamma_p\left(h_c(\gamma, \gamma_i) + 1 \right)},
\end{align}
then we can rewrite \eqref{eq:factorization of product} as
\begin{equation} \label{eq:update from R}
\frac{P_{m_i+k}}{P_{m_i+1}} =
 \left(\tfrac{[z]}{z} \right)^{k-1} \frac{R_i(\left( k - \gamma_{i,c} \right)\tfrac{p}{1-p})}{R_i(\left( 1 - \gamma_{i,c} \right)\tfrac{p}{1-p})} \cdot \prod_{j=1}^{k-1} \left. z\tfrac{f_{i,c}(x+j)}{g_{i,c}(x+j)}\right|_{x=( k - \gamma_{i,c}) \tfrac{p}{1-p}}.
\end{equation}
The terms in \eqref{eq:update from R} not involving $j$ depend on $k$ in a usefully simple way: there exist quantities $c_{i,h}(p) \in \Z/p^{e_i-h} \Z$ for $h=0,\dots,e_i-1$ such that for all $k$,
\begin{equation} \label{eq:precomputed quantities}
 \left(\tfrac{[z]}{z} \right)^{k-1} \frac{R_i(\left( k - \gamma_{i,c}\right)\tfrac{p}{1-p})}{R_i(\left( 1 - \gamma_{i,c} \right)\tfrac{p}{1-p})} P_{m_i+1}\equiv
\sum_{h=0}^{e-1} c_{i,h}(p) \left( (k-\gamma_{i,c}) \tfrac{p}{1-p}\right)^h
\pmod{p^{e_i}}.
\end{equation}

\begin{remark}
In lieu of \eqref{eq:update from R} one could try to use the factorization
\begin{equation}
\frac{P_{m_i+k}}{P_{m_i+1}} =
\prod_{j=1}^{k-1} \left. z\tfrac{f_{i,c}(x+j)}{g_{i,c}(x+j)}\right|_{x=( 1- \gamma_{i,c}) \tfrac{p}{1-p}} \cdot  \left(\tfrac{[z]}{z} \right)^{k-1} \cdots,
\end{equation}
but this does not achieve the requisite decoupling; in particular the second factor does not admit a useful representation in terms of $k$.
\end{remark}

Step 3 of Algorithm~\ref{alg:Hp overall} is to compute the $c_{i,h}(p)$ following the approach of \cref{subsec:compute Pmi}. First,
use \eqref{eq:gamma_feq}, \cref{eq:not div by p}, and \eqref{eq:precomputation result not normalized},
now noting that $h_c(\gamma, \gamma_i)+1 \in (0,2]$, to obtain a representation in the form $c \exp(s(x))$ of
\begin{equation} \label{eq:gamma product series mi1}
\gammaprod \Gamma_p\left(x+h_c(\gamma, \gamma_i) +1 \right) \pmod{p^{e_i}}.
\end{equation}
Then compute \eqref{eq:precomputed quantities} as $z^{m_i+1} \pmod{p^{e_i}}$, times \eqref{eq:mult lift serie} evaluated at $m = m_i + k$, times \eqref{eq:gamma product series mi1} evaluated at $x = (k-\gamma_{i,c})\tfrac{p}{1-p}$, divided by \eqref{eq:gamma base}.

\subsection{Form of the matrix product}
\label{subsec:form of product}

We now perform Step 4 of Algorithm~\ref{alg:Hp overall}, using  \eqref{eq:precomputed quantities} to express the desired sum in terms of a matrix product. Rewrite \eqref{eq:update from R} as
\begin{align} \label{eq:update from R factored}
P_{m_i+k}& \equiv \sum_{h=0}^{e_i-1} c_{i,h}(p) \left( (k-\gamma_{i,c}) \tfrac{p}{1-p} \right)^h \cdot \prod_{j=1}^{k-1} \left. z\tfrac{f_{i,c}(x+j)}{g_{i,c}(x+j)}\right|_{x=( k - \gamma_{i,c}) \tfrac{p}{1-p}} \\
 \label{eq:computed term}
&\equiv \sum_{h_1=0}^{e_i-1}  \sum_{h_2=h_1}^{e_i-1} c_{i,h_1}(p)  Q_{h_1,h_2}(k) \left(\tfrac{p}{1-p} \right)^{h_2} \pmod{p^{e_i}},
\end{align}
where using the notation $f^{[h]}$ from \cref{notation:coeff_notation}
we write
\begin{align}
\label{eq:matrix entries}
Q_{h_1,h_2}(k) &\colonequals
( k - \gamma_{i,c})^{h_2} \left(\prod_{j=1}^{k-1} z \tfrac{f_{i,c}(x+j)}{g_{i,c}(x+j)}\right)^{[h_2-h_1]}.
\end{align}
This prompts us to define the block matrix $A_{i,c}(k)$ with $e_i \times e_i$ blocks as follows, where $\delta_{h_1,h_2}$ is the Kronecker delta and the scalar is chosen to clear denominators
(see \cref{R:clear denominators}):
\begin{equation} \label{eq:matrix formula}
A_{i,c}(k) \colonequals \mbox{(scalar)} \begin{pmatrix}
\delta_{h_1,h_2} & 0 \\
\overline{\sigma}_i (k-\gamma_{i,c})^{e_i-h_2} \delta_{h_1,h_2} &
\left(z\tfrac{f_{i,c}(x+k)}{g_{i,c}(x+k)}\right)^{[h_1-h_2]}
\end{pmatrix};
\end{equation}
note that the factors of $k-\gamma_{i,c}$ appear in the bottom left rather than the bottom right.
As promised, we apply \cref{thm:remainder tree} and \cref{example:paradigm} to compute the products $S_i(p)$ as in \eqref{eq:tilde Sip}; writing $S_i(p)$ as a block matrix in the notation of \eqref{eq:block decomposition}, we have
\begin{align}
(\Delta^{-1} \Sigma)_{h_1,h_2} &= \overline{\sigma}_i \sum_{k=1}^{m_{i+1}-m_i-1} Q_{e_i-h_1,e_i-h_2}(k).
\end{align}
In terms of the $1 \times e_i$ row vectors $v,w$ given by
\begin{equation}
v_j := c_{i,e_i-j}(p), \qquad w_j := (\tfrac{p}{1-p})^{e_i-j},
\end{equation}
by \eqref{eq:computed term} we have
\begin{align} \label{eq:Sip21 second form}
\overline{\sigma}_i \sum_{m=m_i+1}^{m_{i+1}-1} P_{m} &\equiv \frac{1}{\Delta_{e_i,e_i}}
(v \Sigma w^T)_{11}.
\end{align}

\begin{remark} \label{R:select rows}
As in \cref{R:select row}, in practice we achieve a constant factor speedup by computing not $S_i(p)$ but $VS_i(p)$ where $V$ is the matrix consisting of the last $e_i+1$ rows of the $2e_i \times 2e_i$ identity matrix.
\end{remark}

\begin{remark}
According to \eqref{eq:Sip21 second form}, for $j=1,\dots,e_i$ we only need to compute columns $j$ and $e_i+j$ of $S_i(p)$ modulo $p^j$.
However, in practice there seems to be little overhead incurred by computing all of $S_i(p)$ modulo $p^{e_i}$.
\end{remark}

\begin{remark} \label{R:clear denominators}
The scalar in \eqref{eq:matrix formula}
can be bounded by $b_i^{e_i-1} g_{i,c}(k) \rad(g_{i,c}(k))^{e_i-1}$,
where $\rad(g_{i,c}(k))$ is the radical of $g_{i,c}(k)$.
The latter arises from clearing denominators in the series expansion of $g_{i,c}(x+k)^{-1} \pmod{x^{e_i}}$; e.g., for $e_i=2$,
\begin{equation}
A_{i,c}(k) = \mbox{(scalar)} \begin{pmatrix} 1 & 0 & 0 & 0 \\
0 & 1 & 0 & 0 \\
\overline{\sigma}_i(k-\gamma_{i,c}) & 0 & z \tfrac{f_{i,c}(k)}{g_{i,c}(k)} & 0 \\
0 & \overline{\sigma}_i &  z \tfrac{f'_{i,c}(k)}{g_{i,c}(k)} - z \tfrac{f_{i,c}(k) g'_{i,c}(k)}{g_{i,c}(k)^2} &  z \tfrac{f_{i,c}(k)}{g_{i,c}(k)}
\end{pmatrix}.
\end{equation}
It is tempting to try to restructure the computation so that the top left corner of $A_{i,c}(k)$ is used to track the product over $g_{i,c}(x+k)$. This is complicated by the need to sum $\left(l-\gamma_{i,c}\right)^{h_1} \left( \prod_{j=1}^k z \tfrac{f_{i,c}(x+j)}{g_{i,c}(x+j)}\right)^{[h_2]}$ for $h_1 > h_2$ to incorporate the $c_{i,h}(p)$; it would also mean retaining all of the rows of the product, rather than only $e_i+1$ of them as in \cref{R:select rows}.

In any case, the dependence on $\rad(g_{i,c}(k))$ means that our algorithm performs much better in cases where $\beta$ has many repeated entries. In particular, it is sometimes very profitable to swap $\alpha$ and $\beta$ when possible (see \cref{rmk:swap alpha beta}).
\end{remark}

\subsection{Complexity estimates}
\label{subsec:complexity}

We conclude by analyzing the complexity of Algorithm~\ref{alg:Hp overall}. This will imply \cref{T:main} by taking $e = \lceil (w+1)/2 \rceil \leq r-1$ and invoking \cref{rmk:sufficient precision}.

We first note that $\varphi(b_i) \leq r$ and so\footnote{By \cite[Theorem~15]{rosser-schoenfeld}, $\varphi(n) \geq \frac{n}{2 \log \log n + 3/\log \log n}$ for $n > 2$.} $b_i = O(r \log \log r)$. From this we see that the sum of all integers that occur as $\lcm(b_i, b_j)$ for some $i,j$ is $O(r^2 (\log \log r)^2)$.
Applying \cref{T:Gamma expansions}, we bound the time complexity of Step 1 by
\[
O(r^2 (\log \log r)^2 e^2 X (\log X)^3 + r^2 (\log \log r)^2 e^4 (\log r + \log e) X)
\]
and the space complexity by
\[
O(e X (\log X)^2 + r^2 (\log \log r)^2 e^2 X).
\]

Steps 2 and 3 include no amortized steps, so their space costs are negligible.
For each of $O(X/\log X)$ primes $p$, we perform $O(r+e^2)$ operations in $\Z/(p^e)$. At a time cost of $O(e^2 (\log X)^2)$ apiece per \cref{R:quadratic time}, this runs to $O((re^2 + e^4) X \log X)$.

Step 4 is dominated by $O(r)$ applications
of \cref{thm:remainder tree} via \cref{example:paradigm}.
Since the matrix $A_{i,c}(k)$ has size $2e_i\times 2e_i$
and its entries have degree $O(e_i r)$
(the factor of $e_i$ coming from \cref{R:clear denominators}),
we bound the time and space costs by
\[
O(e^3 r^2 X (\log X)^3)
\quad \mbox{and} \quad
O(e^3 r^2 X (\log X)^2)
\]
provided that $e = O(r)$, $r = O(\log X)$, and the bitsize of $z$ is $O(er \log X)$. Combining these estimates (and replacing $e$ with $r$) yields \cref{T:main}.

\section{Implementation notes and sample timings}
\label{sec:timings}

We have implemented Algorithm~\ref{alg:Hp overall} in \Sage,
using Cython for certain inner loops and for a wrapper to Sutherland's rforest C library.
While it would be natural to implement a multithreaded approach in certain steps (particularly the calculations not amortized over $p$), we have not implemented this.

We report timings in \cref{table:timings} and divide our algorithm into three phases:
\begin{enumerate}
\item
Step 1 of Algorithm~\ref{alg:Hp overall}. This phase is independent of $z$ and depends only mildly on $\alpha$ and $\beta$.
\item
Steps 2 and 3 of Algorithm~\ref{alg:Hp overall}, in both cases assuming $z=1$. This phase includes no amortized steps and is independent of $z$.
\item
Step 4 of Algorithm~\ref{alg:Hp overall},
plus adjustment of the results of Steps 2 and 3 to account for $z$.
\end{enumerate}

When possible, we also include comparison timings with the built-in functions for computing $\Hp$ in \Sage{} and \Magma{}, denoted ``\Sage($p$)'' and ``\Magma($p$)'' in the table. These are not amortized, and so these runtimes are quasilinear in $X^2$ rather than $X$. These builtin functions also compute Frobenius traces at higher prime powers $q$ (see \cite[(2.22)]{costa-kedlaya-roe20} for the analogous formula), which are needed for tabulation of $L$-functions when $r > 2$ (see \cref{sec:tabulation}); these are denoted
``\Sage($q$)'' and ``\Magma($q$)'' in the table, and the runtimes are quasilinear in $X^{3/2}$.

For $e=1$, we also include a comparison with the ``chained product'' approach of \cite{costa-kedlaya-roe20} (see \cref{R:chained product}). The two approaches perform comparably for individual calculations, but for bulk calculations Algorithm~\ref{alg:Hp overall} is clearly superior.

We use the following hypergeometric data (computing the weight $w$ as in \eqref{eq:mot_wt}):
\begin{center}
\begin{tabular}{ccccc}
$e$ & $r$ & $w$ & $\alpha$ & $\beta$ \\
\hline \\[-6pt]
$1$ & $2$ & $1$ & $(\tfrac{1}{4}, \tfrac{3}{4})$ & $(\tfrac{1}{6}, \tfrac{5}{6})$ \\[2pt]
$1$ & $4$ & $1$ & $(\tfrac{1}{10}, \tfrac{3}{10}, \tfrac{7}{10}, \tfrac{9}{10})$ & $(\tfrac{1}{6}, \tfrac{1}{6}, \tfrac{5}{6}, \tfrac{5}{6})$ \\[2pt]
$2$ & $4$ & $3$ & $(\tfrac{1}{4}, \tfrac{1}{3}, \tfrac{2}{3}, \tfrac{3}{4})$ & $(\tfrac{1}{6}, \tfrac{1}{6}, \tfrac{5}{6}, \tfrac{5}{6})$ \\[2pt]
$3$ & $6$ & $5$ & $(\tfrac{1}{5}, \tfrac{2}{5}, \tfrac{1}{2}, \tfrac{1}{2}, \tfrac{3}{5}, \tfrac{4}{5})$ & $(\tfrac{1}{6}, \tfrac{1}{6}, \tfrac{1}{6}, \tfrac{5}{6}, \tfrac{5}{6}, \tfrac{5}{6})$ \\[2pt]
$4$ & $8$ & $7$ & $(\tfrac{1}{5}, \tfrac{1}{3}, \tfrac{2}{5}, \tfrac{1}{2}, \tfrac{1}{2}, \tfrac{3}{5}, \tfrac{2}{3}, \tfrac{4}{5})$ &
$(\tfrac{1}{6}, \tfrac{1}{6}, \tfrac{1}{6}, \tfrac{1}{6}, \tfrac{5}{6}, \tfrac{5}{6}, \tfrac{5}{6}, \tfrac{5}{6})$ \\
\end{tabular}
\end{center}
We take the specialization point $z = \tfrac{314}{159}$; note that in \Magma{} one must input $z^{-1}$ instead of $z$.

Timings are reported in 5.1GHz Intel i9-12900K core-seconds, running \Sage{} 10.1 and \Magma{} 2.28-5.
Memory usage was limited to 20GB.

\begin{table}[p]
\small
\caption{Timings, see \S \ref{sec:timings} for explanation.}
\label{table:timings}
\begin{tabular}{c|ccccccccccc}
    $e=1, r=2$ & $2^{15}$ & $2^{16}$ & $2^{17}$ & $2^{18}$ & $2^{19}$ & $2^{20}$ & $2^{21}$ & $2^{22}$ & $2^{23}$ & $2^{24}$ & $2^{25}$\\
\hline
\cite{costa-kedlaya-roe20} & 0.06 & 0.11 & 0.21 & 0.48 & 1.20 & 2.84 & 6.80 & 16.0 & 37.6 & 88.5 & 226
\\
\hline
Phase (1) & 0.08 & 0.13 & 0.24 & 0.49 & 0.98 & 2.26 & 8.15 & 17.9 & 32.5 & 65.6 & 157
\\
Phase (2) & 0.03 & 0.05 & 0.10 & 0.19 & 0.38 & 0.71 & 3.20 & 5.54 & 9.44 & 17.1 & 28.2
\\
Phase (3) & 0.04 & 0.07 & 0.12 & 0.24 & 0.48 & 1.02 & 2.30 & 5.44 & 11.2 & 24.8 & 63.3
\\
\hdashline
Total & 0.15 & 0.25 & 0.46 & 0.92 & 1.84 & 3.98 & 13.6 & 28.9 & 53.1 & 107 & 248
\\
\hline
\Sage ($p$) & 40.6 & 150 & 572
\\
\Magma ($p$) & 36.7 & 147 & 604
\\
\hline
\\
$e=1, r=4$ & $2^{13}$ & $2^{14}$ & $2^{15}$ & $2^{16}$ & $2^{17}$ & $2^{18}$ & $2^{19}$ & $2^{20}$ & $2^{21}$ & $2^{22}$ & $2^{23}$\\
\hline
\cite{costa-kedlaya-roe20} & 0.10 & 0.15 & 0.24 & 0.39 & 0.82 & 1.86 & 4.28 & 10.3 & 24.7 & 55.0 & 127
\\
\hline
Phase (1) & 0.06 & 0.08 & 0.11 & 0.19 & 0.34 & 0.66 & 1.34 & 3.18 & 9.94 & 17.6 & 38.6
\\
Phase (2) & 0.02 & 0.03 & 0.06 & 0.11 & 0.24 & 0.45 & 0.87 & 1.67 & 3.42 & 6.85 & 16.6
\\
Phase (3) & 0.03 & 0.04 & 0.07 & 0.14 & 0.25 & 0.54 & 1.17 & 2.73 & 6.32 & 14.3 & 31.9
\\
\hdashline
Total & 0.11 & 0.16 & 0.25 & 0.44 & 0.83 & 1.65 & 3.38 & 7.58 & 19.7 & 38.7 & 87.1
\\
\hline
\Sage ($p$) & 4.33 & 15.5 & 56.0 & 223 & 879
\\
\Magma ($p$) & 3.11 & 11.8 & 44.9 & 181 & 749
\\
\hline
\Sage ($q$) & 0.03 & 0.08 & 0.20 & 0.46 & 1.34 & 3.61 & 9.54 & 28.7 & 69.1 & 193 & 534
\\
\Magma ($q$) & 0.10 & 0.28 & 0.73 & 1.76 & 5.62 & 15.2 & 39.8 & 111 & 325
\\
\hline
\\
$e=2, r=4$ & $2^{14}$ & $2^{15}$ & $2^{16}$ & $2^{17}$ & $2^{18}$ & $2^{19}$ & $2^{20}$ & $2^{21}$ & $2^{22}$ & $2^{23}$ & $2^{24}$\\
\hline
\hline
Phase (1) & 0.12 & 0.23 & 0.41 & 3.28 & 1.59 & 3.33 & 10.7 & 19.4 & 42.9 & 94.8 & 203
\\
Phase (2) & 0.06 & 0.11 & 0.21 & 0.37 & 0.68 & 1.27 & 2.36 & 4.39 & 8.57 & 17.2 & 37.9
\\
Phase (3) & 0.08 & 0.14 & 0.21 & 0.47 & 1.00 & 2.19 & 5.04 & 12.3 & 28.6 & 67.7 & 147
\\
\hdashline
Total & 0.26 & 0.48 & 0.84 & 4.12 & 3.27 & 6.79 & 18.0 & 36.1 & 80.0 & 179 & 388
\\
\hline
\Sage ($p$) & 119 & 469
\\
\Magma ($p$) & 17.3 & 65.7 & 454
\\
\hline
\Sage ($q$) & 0.11 & 0.23 & 0.48 & 1.40 & 3.67 & 9.58 & 26.1 & 68.4 & 188 & 518
\\
\Magma ($q$) & 0.32 & 0.86 & 2.93 & 10.3 & 29.2 & 81.1 & 249 & 762
\\
\hline
\\
$e=3, r=6$ & $2^{13}$ & $2^{14}$ & $2^{15}$ & $2^{16}$ & $2^{17}$ & $2^{18}$ & $2^{19}$ & $2^{20}$ & $2^{21}$ & $2^{22}$ & $2^{23}$\\
\hline
\hline
Phase (1) & 0.16 & 0.30 & 0.61 & 1.02 & 2.12 & 4.46 & 10.0 & 25.0 & 54.1 & 123 & 283
\\
Phase (2) & 0.05 & 0.08 & 0.16 & 0.38 & 0.59 & 1.26 & 2.33 & 4.69 & 8.67 & 17.2 & 33.8
\\
Phase (3) & 0.11 & 0.23 & 0.32 & 0.72 & 1.98 & 4.60 & 11.4 & 27.6 & 71.4 & 158 & 374
\\
\hdashline
Total & 0.31 & 0.62 & 1.09 & 2.11 & 4.69 & 10.3 & 23.8 & 57.3 & 134 & 298 & 691
\\
\hline
\Sage ($p$) & 54.9 & 210 & 793
\\
\Magma ($p$) & 11.2 & 42.8 & 162 & 626 & 2437
\\
\hline
\Sage ($q$) & 0.06 & 0.12 & 0.28 & 0.57 & 1.64 & 4.31 & 11.3 & 30.7 & 165 & 869 & 2825
\\
\Magma ($q$) & 0.27 & 0.68 & 1.98 & 4.37 & 14.5 & 37.2 & 100 & 291 & 907 & 2951
\\
\hline
\\
$e=4, r=8$ & $2^{13}$ & $2^{14}$ & $2^{15}$ & $2^{16}$ & $2^{17}$ & $2^{18}$ & $2^{19}$ & $2^{20}$ & $2^{21}$ & $2^{22}$ & $2^{23}$\\
\hline
\hline
Phase (1) & 0.45 & 0.72 & 1.23 & 2.35 & 4.86 & 10.7 & 23.5 & 57.9 & 130 & 304 & 697
\\
Phase (2) & 0.09 & 0.17 & 0.32 & 0.60 & 1.24 & 2.33 & 4.38 & 8.38 & 16.8 & 32.4 & 65.5
\\
Phase (3) & 0.19 & 0.40 & 0.67 & 1.78 & 4.47 & 10.9 & 26.1 & 63.9 & 163 & 365 & 865
\\
\hdashline
Total & 0.74 & 1.28 & 2.22 & 4.72 & 10.6 & 23.9 & 54.0 & 130 & 310 & 702 & 1629
\\
\hline
\Sage ($p$) & 76.5 & 291 & 1088
\\
\Magma ($p$) & 12.8 & 49.4 & 189 & 745 & 3789
\\
\hline
\Sage ($q$) & 0.07 & 0.20 & 0.31 & 1.74 & 7.75 & 24.6 & 66.6 & 184 & 494
\\
\Magma ($q$) & 0.30 & 0.81 & 2.32 & 5.09 & 21.2 & 56.4 & 155 & 463

\end{tabular}
\end{table}

\section{Tabulation of \texorpdfstring{$L$}{L}-functions}
\label{sec:tabulation}

As mentioned in the introduction, the broader context for our work is the desire to tabulate hypergeometric $L$-functions at scale in LMFDB. This problem naturally breaks down as follows.
\begin{itemize}
\item
For a fixed hypergeometric $L$-function of motivic weight $w$, we may compute $p$-Frobenius traces for small $p$ using \eqref{eq:Hq} and for larger $p$ up to some cutoff $X$
using our algorithm,
in both cases excluding tame and wild primes.
\item
For prime powers $q \leq X$, excluding powers of tame and wild primes, as noted in \cref{sec:timings} we may use \Sage{} or \Magma{} to compute the $q$-Frobenius trace. (We may also skip $q = p^f$ for $f > r/2$ by virtue of the local functional equation, but this is a minor point because the dominant case is $f=2$.)  While it is likely possible to to reduce the complexity from $X^{3/2}$ to $X$ either by adapting our present approach (as suggested already in \cite{costa-kedlaya-roe20})
or implementing the algorithm indicated in \cite{kedlaya-19},
the timings in \cref{sec:timings} diminish the urgency of this.
\item
\Magma{} can compute Euler factors and conductor exponents at tame primes (modulo a tractable conjecture). The computational difficulty is negligible.\footnote{Note that identifying \emph{all} tame primes requires an integer factorization, but here we only need to know which primes $\leq X$ are tame.}
\item
For Euler factors and conductor exponents at wild primes, even the conjectural picture remains incomplete, but see \cite{roberts-rodriguez-villegas} for a partial description.
Given enough Fourier coefficients at good primes, one can empirically verify a complete guess for the conductor, the global root number, and all bad Euler factors using the approximate functional equation, as in \cite{dokchitser} or \cite{farmer-koutsoliotas-lemurell}.
\end{itemize}

Work in this direction is ongoing, and we expect to have many examples of hypergeometric $L$-functions available in LMFDB in the near future.

\printbibliography

\end{document}